\documentclass[11pt,a4paper,twoside]{article}
\usepackage{amsthm,amsfonts,amsmath}

\newtheorem{Theorem}{Theorem}
\newtheorem{Definition}{Definition}
\newtheorem{Lemma}{Lemma}
\newtheorem{Proposition}{Proposition}
\newtheorem{Remark}{Remark}

\author{Vladimir Rovenski\footnote{Department of Mathematics, University of Haifa, 3498838 Haifa, Israel.
       \newline
       e-mail: \texttt{vrovenski@univ.haifa.ac.il} } }

\title{Characterization of Sasakian manifolds}

\begin{document}

\date{}

\maketitle

\begin{abstract}
Weak contact metric manifolds, i.e., the~linear complex structure on the contact distribu\-tion is replaced by a nonsingular skew-symmetric tensor,
defined by the author and R.\,Wolak, allowed us to take a new look at the theory of contact manifolds.
In this paper, we continue our study \cite{rov-2023,rov-2023c} of a structure of this type, called a weak nearly Sasakian structure,
and prove two theorems characterizing Sasakian manifolds.
Our~main result generalizes the theorem by A.\,Nicola--G.\,Dileo--I.\,Yudin (2018)
and provides a new criterion for a weak almost contact metric manifold to be Sasakian.


\end{abstract}


\section{Introduction}
\label{sec:00-ns}

Contact Riemannian geometry plays an important role in both pure mathematics and theoretical physics.
It considers a $(2n + 1)$-dimensional smooth manifold $M$ equipped with an almost contact metric structure $(\varphi,\xi,\eta,g)$, i.e.,
$g$ is a Riemannian metric with the Levi-Civita connection $\nabla$,
$\varphi$ is a $(1,1)$-tensor, $\xi$ is a unit vector field and $\eta$ is its dual 1-form ($\eta(\xi)=1$)~satisfying
 $g(\varphi X,\varphi Y)= g(X,Y) -\eta(X)\,\eta(Y)$ for any $X,Y\in\mathfrak{X}_M$,
and $\mathfrak{X}_M$ is the space of smooth vector fields on $M$.
An important class of almost contact metric manifolds $M^{\,2n+1}(\varphi,\xi,\eta,g)$ consists of Sasakian manifolds characterized by the equality, see \cite[Theorem~6.3]{blair2010riemannian},
\begin{equation}\label{E-nS-Sas}
 (\nabla_X\,\varphi)Y=g(X,Y)\,\xi -\eta(Y)X,\quad X,Y\in\mathfrak{X}_M .
\end{equation}
D.~Blair, D.~Showers and Y.~Komatu \cite{blair1976} defined nearly Sasakian structure $(\varphi,\xi,\eta,g)$
using a condition similar to \eqref{E-nS-Sas} for the symmetric part of $\nabla\varphi$:
\begin{equation}\label{E-nS-02}
 (\nabla_Y\,\varphi)Y = g(Y,Y)\,\xi -\eta(Y)Y,\quad Y\in\mathfrak{X}_M,
\end{equation}
and showed that a normal nearly Sasakian structure is Sasakian and hence is contact.
Nearly Sasakian structure is a counterpart of the nearly K\"{a}hler structure $(J,g)$ defined by A.~Gray \cite{G-70} by the condition that the symmetric part of $\nabla J$ vanishes, in contrast to the K\"{a}hler case where $\nabla J=0$.
An~example of a nearly Sasakian (non-Sasakian) manifold is the sphere $S^{\,5}$ endowed with an almost contact metric structure induced by the almost Hermitian structure of~$S^{\,6}$.

The curvature and topology of nearly Sasakian manifolds have been studied by several authors, e.g.,
Z.~Olszak, see \cite{Ol-1979,Ol-1980}, B. Cappelletti-Montano, G. Dileo, A.D. Nicola and I. Yudin, see~\cite{C-MD-2016,NDY-2018}.
Any nearly Sasakian manifold of dimension greater than 5 is Sasakian, see \cite{NDY-2018},
and any 5-dimensional nearly Sasakian manifold has Einstein metric of positive scalar curvature, see~\cite{C-MD-2016};
and a 3-dimensional nearly Sasakian manifold is Sasakian, see~\cite{Ol-1980}.

In \cite{RP-2,RWo-2,rov-117} and \cite[Section~5.3.8]{Rov-Wa-2021}, we introduced metric structures
that genera\-lize the almost contact, Sasakian, etc. metric structures.
In~\cite{rov-2023,rov-2023c} we investigated new structures of this type, called {weak nearly Sasakian structure}
and {weak nearly K\"{a}hler structure}, and asked the question: \textit{under what conditions are weak nearly Sasakian manifolds Sasakian}?

These so-called ``weak" structures (the~li\-near complex structure on the contact distribution is replaced by a nonsingular skew-symmetric tensor)
made it possible to take a new look at the classical theory and find new applications,
e.g., in Sasakian twistor spinors theory, see E.C.~Kim~\cite{K-2023}.
For an oriented Riemannian manifold, its twistor space can be seen as the space of linear complex structures on the tangent space
that are compatible with the metric $g$ and orientation.
Some authors, for example, A.C. Herrera in \cite{H-2022}, studied the problem of finding parallel skew-symmetric 2-tensors different from almost-complex structures on a Riemannian manifold and classified such tensors.
We~believe that theories like twistor spinors can be naturally extended by considering non-singular skew-symmetric tensors instead of linear complex~structures.

In this article, we continue our study of the geometry of weak nearly Sasakian manifolds and
find new criterions for a weak almost contact metric manifold to be Sasakian.
Section~\ref{sec:01-ns}, following the introductory Section~\ref{sec:00-ns}, reviews the basics of weak almost contact manifolds
and auxiliary results on the geometry of weak almost Sasakian structure with conditions \eqref{E-nS-10}--\eqref {E-nS-04c}.
Section~\ref{sec:04-ns} presents some properties of the curvature tensor and covariant derivatives of tensors $\varphi$ and $h$.
In Theorem~\ref{T-0.1}, we consider weak almost contact metric manifolds with the condition \eqref{E-nS-10}
and characterize Sasakian manifolds in this class using the property \eqref{E-nS-Sas}.
In Theorem~\ref{Th-4.5} (generalizing Theorem~3.3 from \cite{NDY-2018}, see also Theorem~4.9 from~\cite{C-MD-2019}) we characterize Sasakian manifolds of dimension greater than five in the class of weak almost contact metric manifolds with conditions \eqref{E-nS-10}--\eqref{E-nS-04c}.
In~the proof of Theorem~\ref{Th-4.5} we also use our result from \cite{rov-2023c} (see Theorem~\ref{T-2.2}) that
under certain conditions the weak nearly Sasakian structure foliates into two types of totally geodesic foliations.
In~Section~\ref{sec:app}, we prove two lemmas (formulated in Section~\ref{sec:04-ns}) that allow us to obtain the main results.
Our~proofs use the properties of new tensors, as well as classical constructions.

\section{Preliminaries}
\label{sec:01-ns}

A {weak almost contact metric structure} on a smooth manifold $M^{2n+1}\ (n>0)$ is a set $(\varphi,Q,\xi,\eta, g)$,
where $\varphi$ is a $(1,1)$-tensor, $\xi$ is a vector field, $\eta$ is a dual 1-form ($\eta(\xi)=1$),
$g$ is a Riemannian metric and $Q$ is a nonsingular $(1,1)$-tensor on $M$, satisfying, see \cite{RP-2,RWo-2},
\begin{align}\label{E-nS-2.2}
g(\varphi X,\varphi Y)= g(X,Q\,Y) -\eta(X)\,\eta(Y),\quad X,Y\in\mathfrak{X}_M .
\end{align}
Assume that a $2n$-dimensional distribution $\ker\eta$ is $\varphi$-invariant, i.e., $\varphi(\ker\eta)\subset\ker\eta$,
as in the classical theory \cite{blair2010riemannian}, where $Q={\rm id}_{\,TM}$.
By this and \eqref{E-nS-2.2}, $\ker\eta$ is $Q$-invariant and the following~hold:
\begin{align*}
 \varphi\,\xi=0,\quad \eta\circ\varphi=0,\quad \eta\circ Q=\eta,\quad [Q,\,\varphi]:={Q}\circ\varphi - \varphi\circ{Q}=0.
\end{align*}
For a weak almost contact metric structure, $\varphi$ is skew-symmetric and $Q$ is self-adjoint and positive definite.
Setting $Y=\xi$ in \eqref{E-nS-2.2}, we obtain $\eta(X)=g(\xi, X)$.
We get a {weak contact metric structure} is the following is true:
\[
 d\eta(X,Y)=g(X,\varphi Y),\quad X,Y\in\mathfrak{X}_M,
\]
where the exterior derivative $d\eta$ of $\eta$ is given~by
 $d\eta(X,Y) = \frac12\,\{X(\eta(Y)) - Y(\eta(X)) - \eta([X,Y])\}$, for example,~\cite{blair2010riemannian}.
A~1-form $\eta$ on a manifold $M^{\,2n+1}$ is said to be {contact} if $\eta\wedge (d\eta)^n\ne0$, e.g.,~\cite{blair2010riemannian}.
For a weak contact metric structure $(\varphi,Q,\xi,\eta,g)$, the 1-form $\eta$ is contact, see \cite[Lemma~2.1]{rov-2023b}.

\begin{Definition}[see \cite{rov-2023}]\rm
A weak almost contact metric manifold $M^{\,2n+1}(\varphi,Q,\xi,\eta,g)$ is called \textit{weak nearly Sasaki\-an}~if \eqref{E-nS-02} is true, or, equivalently,
\begin{equation}\label{E-nS-00b}
 (\nabla_Y\,\varphi)Z + (\nabla_Z\,\varphi)Y = 2\,g(Y,Z)\,\xi -\eta(Z)Y -\eta(Y)Z,\qquad Y,Z\in\mathfrak{X}_M.
\end{equation}
\end{Definition}

In addition to \eqref{E-nS-00b}, the following two conditions, which are automatically satis\-fied by almost contact metric manifolds,
play a significant  role in this paper:
\begin{align}\label{E-nS-10}
 & (\nabla_X\,\widetilde Q)\,Y=0,\quad X\in\mathfrak{X}_M,\ Y\in\ker\eta, \\
\label{E-nS-04c}
 & R_{\widetilde Q X,Y}Z\in\ker\eta,\quad X,Y,Z\in\ker\eta ,
\end{align}
where
$\widetilde Q= Q - {\rm id}_{\,TM}$ is a ``small" tensor,  and
 $R_{{X},{Y}}=\nabla_X\nabla_Y -\nabla_Y\nabla_X -\nabla_{[X,Y]}$ for all $X,Y\in\mathfrak{X}_M$,
is the curvature tensor, e.g.,~\cite{CLN-2006}.
From \eqref{E-nS-04c} and the Bianchi identity we get
\[
 R_{\,X,Y}\,\widetilde Q Z\in\ker\eta,\quad X,Y,Z\in\ker\eta.
\]
Using the equalities $(\nabla_\xi\,\varphi)\,\xi = 0$ and $\varphi\,\xi=0$ for a weak nearly Sasakian manifold, we conclude that $\xi$ is a geodesic field ($\nabla_\xi\,\xi=0$).
From \eqref{E-nS-10} and $\nabla_\xi\,\xi=0$, we~get
\begin{equation*}
 \nabla_\xi\,\widetilde Q=0.
\end{equation*}
The $\xi$-curves form a Riemannian geodesic foliation, and $-\nabla\xi$ is its splitting operator (co-nullity tensor), see \cite[Section~1.3.1]{Rov-Wa-2021}.
Taking derivative of the equality $g(\varphi V,Z)=-g(V,\varphi Z)$, we see that the tensor $\nabla_{Y}\,\varphi$ of a weak nearly Sasakian manifold is skew-symmetric:
\begin{equation}\label{E-nS-05e}
 g((\nabla_{Y}\,\varphi) V, Z)=-g((\nabla_{Y}\,\varphi) Z, V).
\end{equation}
The Ricci identity (commutation formula) for the (1,1)-tensor $\varphi$ can be written as (e.g., \cite{CLN-2006})
\begin{equation}\label{E-nS-05}
 g((\nabla^2_{X,Y}\,\varphi)V, Z) - g((\nabla^2_{Y,X}\,\varphi)V, Z) = g(R_{\,X,Y}\,\varphi V, Z) + g(R_{\,X,Y}V, \varphi Z).
\end{equation}
where the second covariant derivative is given by $\nabla^2_{X,Y} = \nabla_{X}\nabla_{Y} - \nabla_{\nabla_{X}Y}$.
Taking covariant derivative of \eqref{E-nS-05e}, we see that $\nabla^2_{X,Y}\,\varphi$ of a weak nearly Sasakian manifold is~skew-symmetric:
\[
 g((\nabla^2_{X,Y}\,\varphi)V, Z)=-g((\nabla^2_{X,Y}\,\varphi)Z, V) .
\]

In the rest of the section, we survey some properties of weak nearly Sasakian manifolds, see~\cite{rov-2023c}.
Define a (1,1)-tensor field $h$ on $M$ (as in the classical case, e.g., \cite{NDY-2018}),
\begin{equation}\label{E-c-01}
  h = \nabla\xi + \varphi.
\end{equation}
We get $\eta\circ h = 0$ and $h(\ker\eta)\subset\ker\eta$. Since $\xi$ is a geodesic vector field, we also get $h\,\xi=0$.
Since $\xi$ is a Killing vector field, see \cite{rov-2023}, and $\varphi$ is skew-symmetric, the tensor $h$ is skew-symmetric.
The~distribution $\ker\eta$ is integrable, $[X,Y]\in\ker\eta\ (X,Y\in\ker\eta)$, if and only if $h=\varphi$,
and in this case, our manifold splits along $\xi$ and $\ker\eta$, i.e., is locally the metric product.

\begin{Lemma}[see \cite{rov-2023c}] \label{L-nS-02}
For a weak nearly Sasakian manifold $M^{\,2n+1}(\varphi,Q,\xi,\eta,g)$ we obtain
\begin{align}
\label{E-nS-01a}
 & (\nabla_X\,h)\,\xi = -h(h-\varphi)X ,\\
\label{E-nS-01c}
 & (\nabla_X\,\varphi)\,\xi =  -\varphi(h-\varphi) X ,\\
\label{E-nS-01b}
 & h\,\varphi + \varphi\,h = -2\,\widetilde Q ,\\
\label{E-nS-01d}
 & h\,Q = Q\,h \quad (h\ {\rm commutes\ with}\ Q).
\end{align}
Moreover,
 $h^2\varphi=\varphi h^2$, $h\varphi^2=\varphi^2 h$, $h^2\varphi^2=\varphi^2 h^2$, etc.
\end{Lemma}

If a weak nearly Sasakian manifold satisfies the condition \eqref{E-nS-04c}, then its contact distribution
is {curvature invariant}, see~\cite{rov-2023c}, i.e.,
\begin{equation}\label{E-nS-04cc}
 R_{\,X,Y}Z\in\ker\eta,\quad X,Y,Z\in\ker\eta .
\end{equation}
For example, the tangent bundle of a totally geodesic submanifold in a Riemannian manifold and the distribution $\ker\eta$ of any 1-form $\eta$ on a real space form satisfy \eqref{E-nS-04cc}.

\begin{Lemma}[see \cite{rov-2023c}]\label{L-nS-04}
For a weak nearly Sasakian manifold $M^{\,2n+1}(\varphi,Q,\xi,\eta,g)$ with the condition
\eqref{E-nS-04c}, we obtain
\begin{align}
\label{E-3.24}
 & R_{\,X,\xi}\,Y = (\nabla_X\,(h-\varphi)\,)\,Y
  = (\nabla_X\,(h-\varphi))Y =  g( (h-\varphi)^2 X, Y)\,\xi - \eta(Y)\,(h-\varphi)^2 X .
\end{align}
In particular, $\nabla_\xi\,h = \nabla_\xi\,\varphi = \varphi h + \widetilde Q$ and
\begin{align}\label{E-3.23b}
 g(R_{\,\xi, X}Y, Z) =  \eta(Y)\,g( (h-\varphi)^2 X , Z) - \eta(Z)\,g( (h-\varphi)^2 X, Y).
\end{align}
\end{Lemma}

A Sasakian manifold is defined as a contact metric manifold such that the tensor
$[\varphi, \varphi] + d\eta\otimes\xi$ vanishes identically,
where the Nijenhuis torsion $[\varphi,\varphi]$ of $\varphi$
is given~by (for example,~\cite{blair2010riemannian})
\begin{align*}
 [\varphi,\varphi](X,Y) = \varphi^2 [X,Y] + [\varphi X, \varphi Y] - \varphi[\varphi X,Y] - \varphi[X,\varphi Y].
\end{align*}

The following statement, generalizing \cite[Lemma~2.1]{Ol-1980}, will be used in the proof of Theorem~\ref{T-0.1}.

\begin{Proposition}[see also \cite{rov-2023c}]\label{Th-4.1}
For a weak nearly Sasakian manifold with the property \eqref{E-nS-10}, the equality $h=0$ holds if and only if the manifold is Sasakian.
\end{Proposition}

For a weak nearly Sasakian manifold satisfying \eqref{E-nS-10} and \eqref{E-nS-04c},
the eigenvalues (and their multiplicities) of the self-adjoint operator $h^2$ are constant, see \cite{rov-2023c}.
Since $h$ is skew-symmetric, the nonzero eigenvalues of $h^2$ are negative, and the spectrum of $h^2$ has the~form
\begin{equation}\label{E-nS-11b}
  Spec(h^2) = \{0, -\lambda_1^2,\ldots -\lambda_r^2\} ,
\end{equation}
where $\lambda_i$ is a positive real number and $\lambda_i\ne \lambda_j$ for $i\ne j$.
If $X$ is a unit eigenvector of $h^2$ with eigenvalue $-\lambda^2_i$, then by  \eqref{E-nS-01b} and \eqref{E-nS-01d},
$X, \varphi X, hX$ and $h\,\varphi X$ are nonzero eigenvectors of $h^2$ with eigenvalue $-\lambda^2_i$.
Denote by $[\xi]$ the 1-dimensional distribution generated by $\xi$,
and by $D_0$ a distribution of the eigenvectors with eigenvalue~0 orthogonal to $\xi$.
Denote by $D_i$ a distribution of the eigenvectors with eigenvalue $-\lambda^2_i$.
Note that
$D_0$ and $D_i\ (i = 1,\ldots, r)$ belong to $\ker\eta$ and are $\varphi$- and $h$- invariant.
The tensors $h$ and $\widetilde Q$ of a weak nearly Sasakian manifold vanish on $D_0$,~\cite{rov-2023c}.

The following statement generalizes \cite[Theorems~3.3 and 3.5]{C-MD-2016}.

\begin{Theorem}[see \cite{rov-2023c}]\label{T-2.2}
Let $M^{\,2n+1}(\varphi,Q,\xi,\eta,g)$ be a weak nearly Sasakian manifold with conditions \eqref{E-nS-10} and \eqref{E-nS-04c},
and let the spectrum of the self-adjoint operator $h^2$ have the form \eqref{E-nS-11b}.
Then, the distribution $[\xi]\oplus D_0$ and each distribution $[\xi]\oplus D_i\ (i = 1,\ldots, r)$ are integrable with totally geodesic leaves.
In particular, the eigenvalue $0$ has multiplicity $2p+1$ for some integer $p\ge0$.
If $p>0$, then

\noindent\ \
$(a)$
the distribution $[\xi]\oplus D_1\oplus\ldots\oplus D_r$ is integrable and defines a Riemannian foliation with totally geodesic~leaves;

\noindent\ \
$(b)$
the leaves of $[\xi]\oplus D_0$ are $(2p+1)$-dimensional Sasakian manifolds.
\end{Theorem}

\section{Main results}
\label{sec:04-ns}

Here, we prove some properties of the curvature tensor and covariant derivatives of the tensors $\varphi$ and $h$,
and prove two theorems characterizing Sasakian manifolds in the class of weak almost contact metric manifolds.

First, we consider weak almost contact metric manifolds with the condition \eqref{E-nS-10} and
characterize Sasakian manifolds in this class using the property \eqref{E-nS-Sas}.

\begin{Theorem}\label{T-0.1}
Let $M(\varphi, Q, \xi, \eta, g)$ be a weak almost contact metric manifold with conditions \eqref{E-nS-Sas} and \eqref{E-nS-10}.
Then $Q={\rm id}_{\,TM}$ and the structure $(\varphi, \xi, \eta, g)$ is Sasakian.
\end{Theorem}

\begin{proof}
Differentiating the equality $g(\varphi\,Y,\xi)=0$ yields
\[
 0=X g(\varphi\,Y,\xi) = g((\nabla_X\,\varphi)Y, \xi) +g(\varphi\,Y, (h-\varphi)X).
\]
Applying to this \eqref{E-nS-Sas} and then \eqref{E-nS-2.2} gives $h\varphi = -\widetilde Q$. Since $h,\varphi$ are skew-symmetric and $\widetilde Q$ is self-adjoint, we also get $\varphi h = -\widetilde Q$. Therefore,
\begin{equation}\label{E-sol-1}
  h\varphi = \varphi h = -\widetilde Q.
\end{equation}
Differentiating \eqref{E-nS-2.2} and using \eqref{E-nS-10}, \eqref{E-sol-1}  and the skew-symmetry of $\nabla_X\,\varphi$, we get
\[
 -\eta(Y)\,g(h X,Z) - \eta(Z)\,g( \varphi X + Q\,(h - \varphi)X, Y) = 0.
\]
Taking $Y=\xi$ in the above equality, we get $g(h X, Z)=0$ for all $X,Z\in\mathfrak{X}_M$.
Therefore, $h=\widetilde Q=0$ and by Proposition~\ref{Th-4.1}, the structure is Sasakian.
\end{proof}

The following statement generalizes Lemma 2.4 in \cite{Ol-1979}.

\begin{Lemma}
For covariant derivatives of the tensor $\varphi$ on a weak nearly Sasakian manifold with the condition \eqref{E-nS-10} we obtain
\begin{align}\label{E-3.29}
 & g((\nabla_X\,\varphi)\varphi Y, Z) = g((\nabla_X\,\varphi)Y, \varphi Z) + \eta(Y)\,g( (h-\varphi) X, Z)
 + \eta(Z)\,g((h-\varphi) X, Q Y) , \\
\label{E-3.30}
\nonumber
 & g((\nabla_{\varphi X}\,\varphi)Y, Z) = g((\nabla_X\,\varphi)\,Y, \varphi Z) - \eta(X)\,g(h Y, Z) - 2\,\eta(Y)\,g(\varphi X, Z) \\
 & \quad\qquad +\,2\,\eta(Z)\,g(\varphi X, Y) - \eta(Z)\,g(Q X, (h-\varphi) Y) .\\
\label{E-3.31}
\nonumber
 & g((\nabla_{\varphi X}\,\varphi)\,\varphi Y, Z)
 = -g((\nabla_X\,\varphi)Y, Q Z) +\eta(X)\,g(Y, h\varphi Z) +\eta(Y)\,g(h\varphi X + \varphi^2 X, Z) \\
 &\quad\qquad +\,\eta(Z)\,g(\varphi(h-\varphi) X, Y)  + \eta(Z)\,g((h-\varphi)\varphi X, Q Y) .
\end{align}
\end{Lemma}

\smallskip\textbf{Proof}.
Differentiating \eqref{E-nS-2.2} and using \eqref{E-nS-01b}, \eqref{E-nS-10} and the skew-symmetry of $\nabla_X\,\varphi$, we get \eqref{E-3.29}.
We obtain \eqref{E-3.30} from \eqref{E-3.29} by the condition~\eqref{E-nS-00b}.
Replacing $X$ by $\varphi X$ in \eqref{E-3.29} gives
\[
  g((\nabla_{\varphi X}\,\varphi)\,\varphi Y, Z)
 = g((\nabla_{\varphi X}\,\varphi)\,Y, \varphi Z) + \eta(Y)\,g((h-\varphi)\varphi X, Z) + \eta(Z)\,g((h-\varphi)\varphi X, Q Y),
\]
from which, using \eqref{E-3.30} and \eqref{E-nS-2.2}, we get \eqref{E-3.31}.
\hfill$\square$

\begin{Lemma}\label{L-R01b}
The curvature tensor of a weak nearly Sasakian manifold satisfies
\begin{align}\label{E-3.4}
\nonumber
 & g(R_{\,\varphi X,Y}Z, V) +g(R_{\,X,\varphi Y}Z, V) +g(R_{\,X,Y}\,\varphi Z, V) +g(R_{\,X,Y}Z, \varphi V) \\
\nonumber
 & = g(Y,V)\,g((h-\varphi)X,Z)-g(X,Y)\,g(Z, (h-\varphi)V) + g(Y,Z)\,g(X, (h-\varphi)V) \\
 &-(1/2)\,g(Z,V)\,g((h-\varphi)X,Y) +(1/2)\,g(X,Z)\,g(Y, (h-\varphi)V)\quad (X,Y,Z,V\in\mathfrak{X}_M).
\end{align}
The curvature tensor of a weak nearly Sasakian manifold with the condition \eqref{E-nS-04c} satisfies
\begin{align}
\label{E-3.6}
\nonumber
 & g(R_{\,\varphi X,\varphi Y}Z, V) = g(R_{\,X,Y}\,\varphi Z, \varphi V) -(1/2)\,\delta(X,Y,Z,V) - g(X,Y)\,g(Z,\widetilde Q V) \\
\nonumber
 &\ +(1/2)\,g(Y,V)\,g(\varphi(h-\varphi)X, Z) - (1/2)\,g(Y,\varphi Z)\,g(X, (h-\varphi)V) \\
\nonumber
 &\ - (1/2)\,g(Y,\varphi V)\,g((h-\varphi)X,Z)  - (1/2)\,g(Y,Z)\,g(X, (h-\varphi)\varphi V) \\
 &\ -(1/4)\,g(X,\varphi Z)\,g(Y, (h-\varphi)V) -(1/4)\,g(X,Z)\,g(Y, (h-\varphi)\varphi V) , \\
\label{E-3.5}
\nonumber
 & g(R_{\,\varphi X,\varphi Y}\,\varphi Z, \varphi V) = g(R_{\,Q X, Q Y}Z, V) - \eta(X)\,g(R_{\,\xi,Q Y}Z, V) + \eta(Y)\,g(R_{\,\xi, Q X}Z, V) \\
\nonumber
 &\ +(1/2)\,g(\varphi Y,V)\,g((h-\varphi)\varphi X,\varphi Z) + (1/2)\,g(\varphi^2 Y, Z)\,g(X, \varphi(h-\varphi) V) \\
\nonumber
 &\ - (1/2) g(\varphi^2 Y,V) g((h-\varphi)\varphi X,Z) + (1/2) g(\varphi Y,Z) g(\varphi X, (h-\varphi)\varphi V) \\
\nonumber
 &\ +(1/4)\,g(\varphi^2 X, Z)\,g(Y, \varphi(h-\varphi)V) +(1/4)\,g(\varphi X,Z)\,g(\varphi Y, (h-\varphi)\varphi V)  \\
 &\ - g(\varphi^2 X, Y) g(Z, \widetilde Q V) +(1/2)\,\delta(\varphi X,\varphi Y,Z,V) ,
\end{align}
where the ``small" tensor $\delta$ is defined by
\begin{equation*}
 \delta(X,Y,Z,V) = g(R_{\,X, Y}\widetilde Q Z, V) +g(R_{\,X, Y}Z, \widetilde Q V) -g(R_{\widetilde Q X, Y}Z, V) -g(R_{\,X, \widetilde Q Y}Z, V) .
\end{equation*}
\end{Lemma}

\begin{Remark}\label{Rem-delta}\rm
The tensor $\delta$ of a weak nearly Sasakian manifold has the following symmetries:
\begin{align*}
 \delta(Y,X, Z,V) = \delta(X,Y, V,Z) = \delta(Z,V, X,Y) = -\delta(X,Y, Z,V) .
\end{align*}
If \eqref{E-nS-04c} is true, then by \eqref{E-3.23b}, we get
\begin{align*}
 & \delta(X,Y,Z,\xi)
 = g(R_{\,X, Y}\,\widetilde Q Z, \xi) -g(R_{\widetilde Q X, Y}Z, \xi) -g(R_{\,X, \widetilde Q Y}Z, \xi) \\
 &\qquad
 = g(R_{\,\xi,\, \widetilde Q Z}Y, X) +g(R_{\,\xi, Z}\,\widetilde Q X, Y) +g(R_{\,\xi, Z}\,X, \widetilde Q Y)
 = 0 ,
\end{align*}
hence, $\delta(\xi,Y,Z,V)=\delta(X,\xi,Z, V)=\delta(X,Y,\xi, V)=\delta(X,Y,Z, \xi)=0$.
\end{Remark}

The following lemma generalizes the result obtained in \cite{Ol-1979}, see also equation (6) in \cite{NDY-2018}.

\begin{Lemma}\label{L-R03}
For a weak nearly Sasakian manifold with the conditions \eqref{E-nS-10} and \eqref{E-nS-04c}, we obtain
\begin{align}\label{E-3.50c}
  g((\nabla_{X}\,\varphi)Y, h Z) = -\eta(X)\,g( (\varphi h^2 + \widetilde Q h) Z, Y) +\eta(Y)\,g( (\varphi h^2 - h +\widetilde Q h) Z, X) .
\end{align}
\end{Lemma}

The proofs of Lemmas~\ref{L-R01b}--\ref{L-R03} are given in Section~\ref{sec:app}.
Next, we generalize \cite[Proposition~3.1]{NDY-2018}.

\begin{Proposition}
For a weak nearly Sasakian manifold with conditions \eqref{E-nS-10} and \eqref{E-nS-04c}, we obtain
\begin{align}
\label{EC-14}
 & (\nabla_X\,\varphi)Y = \eta(X)\,(\varphi  h Y +\widetilde Q Y) -\eta(Y)\,(\varphi h X + Q X) + g(\varphi hX + Q X, Y)\,\xi, \\
\label{EC-15}
 & (\nabla_X\,h)Y = \eta(X)\,(\varphi h Y +\widetilde Q Y) -\eta(Y)\,h (h - \varphi) X + g( h(h - \varphi) X, Y)\,\xi,\\
\label{EC-16}
\nonumber
 & (\nabla_X\,\varphi h)Y = \eta(X)\,(\varphi h^2 Y - hY + \widetilde Q \varphi Y )
  -\eta(Y)\,g(\varphi h^2 X - Q h X + 2\,\widetilde Q\varphi X ) \\
 &\quad + g(\varphi h^2 X - hX + \widetilde Q hX, Y)\,\xi.
\end{align}
\end{Proposition}

\begin{proof}
 From \eqref{E-3.50c} we have
\begin{align*}
 g((\nabla_{X}\,\varphi)Y, h Z) = \eta(X)\,g(\varphi  h Y +\widetilde Q Y, hZ) -\eta(Y)\,g(\varphi h X + X +\widetilde Q X, hZ),
\end{align*}
which is compatible with \eqref{EC-14}. On the other hand,
\begin{align*}
 g((\nabla_{X}\,\varphi)Y, \xi) = - g( Y, (\nabla_{X}\,\varphi)\xi) = g( Y, \varphi(h-\varphi)X)
 = g(QX +\varphi hX, Y) -\eta(X)\,\eta(Y).
\end{align*}
By the above, to prove \eqref{EC-14}, we need only to show (since $\widetilde Q |_{D_0}=0$)
\begin{align}
\label{EC-14b}
 g((\nabla_X\,\varphi)Y, V) = -\eta(Y)\,g(X,V),\quad
 X,Y\in\mathfrak{X}_M,\ V\in {D}_0.
\end{align}
By Theorem~\ref{T-2.2}, the distribution $\bar{D}=[\xi]\oplus D_1\oplus\ldots\oplus D_r$ is integrable with totally geodesic leaves.
If $X,Y\in\bar D$, then both sides of \eqref{EC-14b} vanish.
If $X\in D_0$ and $Y\in\bar D$, then
\begin{align*}
 g((\nabla_X\,\varphi)Y, V) = - g(Y, (\nabla_X\,\varphi)V) = -\eta(Y)\,g(X,V),
\end{align*}
since $[\xi]\oplus D_0$ is integrable with totally geodesic leaves with the induced Sasakian structure on each leaf,
so $(\nabla_X\,\varphi)V = g(X,V)\,\xi - \eta(V)\,X$.
If $X\in\bar D$ and $Y\in D_0$, then
 $g((\nabla_X\,\varphi)Y, V)=-\eta(X)\,g(Y,V)$,
 and applying \eqref{E-nS-00b}, we get
\begin{align*}
 g((\nabla_X\,\varphi)Y, V) = - g((\nabla_Y\,\varphi)X  + \eta(X)Y , V) = 0.
\end{align*}
If we take $X,Y\in D_0$, then \eqref{EC-14} is true because of \eqref{E-nS-Sas} and the orthogonality of vector fields $X,Y,V$ to $\xi$.
This completes the proof of \eqref{EC-14}.
From \eqref{EC-14}, using \eqref{E-3.24}, we obtain \eqref{EC-15}.
Finally, using \eqref{EC-14} and \eqref{EC-15} in the equality
$(\nabla_X\,\varphi h)Y = (\nabla_X\,\varphi)\,hY + \varphi(\nabla_X\,h)Y$ gives
\begin{align*}
& (\nabla_X\,\varphi h)Y = \eta(X)\,(\varphi h^2 Y {-} hY + \widetilde Q \varphi Y + \eta(Y)\,\xi) \\
& - \eta(Y)\,g(\varphi h^2 X - hX + (2\,\varphi - h)\widetilde Q X + \eta(X)\,\xi) + g(\varphi h^2 X {-} hX + \widetilde Q hX, Y)\,\xi ,
\end{align*}
where \eqref{E-nS-2.2} and \eqref{E-nS-01b}  were also used, hence \eqref{EC-16} is true.
\end{proof}

Recall, see \cite[pp.~15--16]{CLN-2006}, that for any 2-form $\beta$ and 1-form $\eta$ we have
\begin{align*}
 & 3\,(\eta\wedge\beta)(X,Y,Z) =
 \eta(X)\,\beta(Y,Z) + \eta(Y)\,\beta(Z,X) + \eta(Z)\,\beta(X,Y) , \\
 & 3\,d\beta(X,Y,Z) = X\beta(Y,Z) {+} Y\beta(Z,X) {+} Z\beta(X,Y) {-} \beta([X,Y],Z) {-} \beta([Z,X],Y) {-} \beta([Y,Z],X) \\
 & \qquad = (\nabla_X\,\beta)(Y,Z) + (\nabla_Y\,\beta)(Z,X) + (\nabla_Z\,h)(X,Y).
\end{align*}
The following proposition will be used in the proof of our main result, Theorem~\ref{Th-4.5}.

\begin{Proposition}[see Proposition~3.2 in \cite{NDY-2018}]\label{Prop-3.2}
Let $\eta$ be a contact 1-form on a smooth manifold $M$ of dimension $2n+1> 5$. Then, the following operator is injective:
\[
 \Upsilon_{d\eta} : \beta\in\Lambda^2(M) \to d\eta\wedge\beta\in\Lambda^4(M),
\]
where $\Lambda^{p}(M)$ is the vector bundle of differential $p$-forms on $M$.
\end{Proposition}

We will generalize Theorem~3.3 in \cite{NDY-2018} on characterization of Sasakian manifolds.

\begin{Theorem}\label{Th-4.5}
Let $M(\varphi, Q, \xi, \eta, g)$ be a weak nearly Sasakian manifold of dimension greater than $5$
with conditions \eqref{E-nS-10} and \eqref{E-nS-04c}. Then $Q={\rm id}_{\,TM}$ and the structure $(\varphi, \xi, \eta, g)$ is Sasakian.
\end{Theorem}

\begin{proof} We consider the following 2-forms:
\begin{align*}
 & \Phi_0(X,Y)=g(h X, Y),\quad \Phi_1(X,Y)=g(\varphi h X, Y),\quad \Phi_2(X,Y)=g(\varphi h^2 X, Y), \\
 & \Psi_0(X,Y)=g(\widetilde Q X, Y),\quad \Psi_1(X,Y)=g(\widetilde Q \varphi X, Y),\quad \Psi_2(X,Y)=g(\widetilde Q h X, Y).
\end{align*}
It is easy to calculate, see \cite{RP-2}:
\[
  3\,d\Phi_0(X,Y,Z) = g((\nabla_X\,h)Z,Y) + g((\nabla_Y\,h)X, Z) + g((\nabla_Z\,h)Y, X).
\]
Applying \eqref{EC-15} to the above, gives
\begin{equation}\label{E-c-03a}
 d\Phi_0 = \eta\wedge(\Phi_1 + \Psi_0).
\end{equation}
We also have
\[
  3\,d\Phi_1(X,Y,Z) = g((\nabla_X\,\varphi h)Z,Y) +g((\nabla_Y\,\varphi h)X, Z) +g((\nabla_Z\,\varphi h)Y, X).
\]
Similarly to \eqref{E-c-03a}, applying \eqref{EC-16} to the above, gives
\begin{equation}\label{E-c-03b}
 d\Phi_1 = \eta\wedge (\Phi_2 - \Phi_0 + \Psi_1).
\end{equation}
Also, using \eqref{E-nS-10}, we find that $d \Psi_0=-\eta\wedge(\Psi_2-\Psi_1)$:
\begin{align*}
 & 3\,d \Psi_0 = g((\nabla_X\, \widetilde Q)Z,Y) +g((\nabla_Y\,\widetilde Q) X, Z) +g((\nabla_Z\,\widetilde Q)Y, X) \\
 & = \eta(Z)\,g((\nabla_X\, \widetilde Q)\,\xi,Y) +\eta(X)\,g((\nabla_Y\,\widetilde Q)\,\xi, Z) +\eta(Y)\,g((\nabla_Z\,\widetilde Q)\,\xi, X) \\
 & = -\eta(Z)\,g(\widetilde Q(h{-}\varphi) X, Y) -\eta(X)\,g(\widetilde Q (h{-}\varphi) Y, Z) -\eta(Y)\,g(\widetilde Q (h{-}\varphi) Z, X)
 = -3\,\eta\wedge(\Psi_2-\Psi_1).
\end{align*}
Using this and \eqref{E-c-03b}, from \eqref{E-c-03a} we obtain
\begin{align}\label{E-Nic-27}
\nonumber
 & 0= d^2\Phi_0 = d\eta\wedge(\Phi_1 + \Psi_0) - \eta\wedge(d\Phi_1 + d \Psi_0) \\
 & = d\eta\wedge(\Phi_1 + \Psi_0) - \eta\wedge\eta\wedge(\Phi_2 - \Phi_0 + 2\,\Psi_1 - \Psi_2)
 = d\eta\wedge(\Phi_1 + \Psi_0) .
\end{align}
If $\dim M > 5$, $\eta$ being a contact form,
the fact that $d\eta\wedge(\Phi_1 + \Psi_0) = 0$, see \eqref{E-Nic-27} and Proposition~\ref{Prop-3.2},
imply $\Phi_1 + \Psi_0 = 0$, i.e., $\varphi h = -\widetilde Q$.
By \eqref{E-nS-01b}, we get $h\varphi = \varphi h = -\widetilde Q$, see \eqref{E-sol-1}.
Using this in \eqref{EC-14} gives the equality \eqref{E-nS-Sas} well known for Sasakian structure.
Using the equality \eqref{E-nS-Sas} in \eqref{E-3.29}, we~get
\[
 -\eta(Y)\,g(h X,Z) +\eta(Z)\,g( (\widetilde Q\,\varphi - Q\,h)X, Y) = 0.
\]
Taking $Y=\xi$ in the above equality, we get $g(h X, Z)=0$ for all $X,Z\in\mathfrak{X}_M$.
We come to the conclusion that $h=\widetilde Q=0$, therefore, by Proposition~\ref{Th-4.1}, the structure is Sasakian.
\end{proof}

\section{Proofs of Lemmas~\ref{L-R01b} and \ref{L-R03}}
\label{sec:app}

Here, in the proofs of Lemmas~\ref{L-R01b}--\ref{L-R03} we use the approach \cite{Ol-1979}.

\smallskip\textbf{Proof of Lemma~\ref{L-R01b}}. This consists of two steps.

\textbf{Step 1}. Some formulas appearing in the proof of \eqref{E-3.4} are also used in the proof of \eqref{E-3.6}.
Differentiating \eqref{E-nS-00b}, we find
\begin{align}\label{EF-nS-01}
 (\nabla^2_{X,Y}\,\varphi)Z + (\nabla^2_{X,Z}\,\varphi)Y = 2\,g(Y,Z)\,(h {-}\varphi) X -g( (h {-}\varphi) X, Z)Y -g( (h{-}\varphi) X, Y)Z.
\end{align}
Applying the Ricci identity \eqref{E-nS-05}, from \eqref{EF-nS-01} and the skew-symmetry of $\nabla^2_{X,Y}\,\varphi$ we get
\begin{align}\label{E-3.7}
\nonumber
 & g(R_{X,Y}Z, \varphi V) -g(R_{X,Y}V, \varphi Z)  + g((\nabla^2_{X,Z}\,\varphi) Y, V) - g((\nabla^2_{Y,Z}\,\varphi)X, V) \\
\nonumber
 & = 2\,g(Y,Z)g( (h{-}\varphi)X, V)-2\,g(X,Z)g( (h{-}\varphi)Y, V) \\
 & - g(Y,V)g( (h{-}\varphi)X, Z) +g(X,V)g( (h{-}\varphi)Y, Z) + 2\,g(Z,V)g( (h{-}\varphi)Y, X).
\end{align}
By Bianchi and Ricci identities, we find
\begin{align}\label{E-3.8}
\nonumber
 & g(R_{X,Y}Z, \varphi V) = -g(R_{Y,Z}X, \varphi V) - g(R_{Z,X}Y, \varphi V) \\
 & = g((\nabla^2_{Y,Z}\,\varphi) V, X) - g((\nabla^2_{Z,Y}\,\varphi)V, X)
   -g(R_{Y,Z}V, \varphi X) - g(R_{Z,X}Y, \varphi V).
\end{align}
Substituting \eqref{E-3.8} into \eqref{E-3.7}, it follows that
\begin{align}\label{E-3.9}
\nonumber
 & g(R_{X,Z}Y, \varphi V) -g(R_{X,Y}V, \varphi Z) - g(R_{Y,Z}V, \varphi X) \\
\nonumber
 & -g((\nabla^2_{Z,Y}\,\varphi)V, X) - g((\nabla^2_{X,Z}\,\varphi)V, Y) = 2\,g((\nabla^2_{Y,Z}\,\varphi)X, V) \\
\nonumber
 & + 2\,g(Y,Z)g( (X-\varphi)X, V) - 2\,g(X,Z)g( (h-\varphi)Y, V) \\
 & + 2\,g(Z,V)g( (h{-}\varphi)Y, X) - g(Y,V)g( (h{-}\varphi)X, Z) + g(X,V)g( (h{-}\varphi)Y, Z).
\end{align}
On the other hand, using \eqref{EF-nS-01} and the Ricci identity \eqref{E-nS-05}, we see that
\begin{align}\label{E-3.10}
 g(R_{X,Z}Y, \varphi V) -g(R_{X,Z}V, \varphi Y)  - g((\nabla^2_{X,Z}\,\varphi)Y, V) +g((\nabla^2_{Z,X}\,\varphi)Y, V) = 0.
\end{align}
Adding \eqref{E-3.10} to \eqref{E-3.9}, we get
\begin{align}\label{E-3.11}
\nonumber
  & 2\,g(R_{X,Z}Y, \varphi V) -g(R_{X,Y}V, \varphi Z)  - g(R_{Y,Z}V, \varphi X) -g(R_{X,Z}V, \varphi Y) \\
\nonumber
 & = 2\,g((\nabla^2_{Y,V}\,\varphi)Z, X) + 2\,g(Y,Z)g( (h-\varphi)X, V)-2\,g(X,Z)g( (h-\varphi)Y, V) \\
 & + 2\,g(Z,V)g( (h-\varphi)Y, X) - g(Y,V)g( (h-\varphi)X, Z) + g(X,V)g( (h-\varphi)Y, Z) .
\end{align}
Swapping $Y$ and $V$ in \eqref{E-3.11}, we find
\begin{align}\label{E-3.12}
\nonumber
 & 2\,g(R_{X,Z}V, \varphi Y) -g(R_{X,V}Y, \varphi Z)  - g(R_{V,Z}Y, \varphi X) -g(R_{X,Z}Y, \varphi V) \\
\nonumber
 & = 2\,g((\nabla^2_{V,Y}\,\varphi)Z, X) + 2\,g(Z,V)g( (h-\varphi)X, Y) - 2\,g(X,Z)g( (h-\varphi)V, Y) \\
 &  + 2\,g(Y,Z)g( (h-\varphi)V, X) - g(Y,V)g( (h-\varphi)X, Z) + g(X,Y)g( (h-\varphi)V, Z) .
\end{align}
Subtracting \eqref{E-3.12} from \eqref{E-3.11}, and using the Bianchi and Ricci identities, we get
\begin{align*}
\nonumber
 & g(R_{\varphi X,Z}Y, V) + g(R_{X,\varphi Z}Y, V)  + g(R_{X,Z}\varphi Y, V) + g(R_{X,Z}Y, \varphi V) \\
\nonumber
 & = g(Z,V)g( (h{-}\varphi)X, Y) - g(X,Z)g( (h{-}\varphi)V, Y) + g(Y,Z)g( (h{-}\varphi)V, X) \\
 &-(1/2)\,g(Y,V)g( (h-\varphi)X, Z) +(1/2)\,g(X,Y)g( (h-\varphi)V, Z) ,
\end{align*}
which (by replacing $Z$ and $Y$) gives \eqref{E-3.4}.

\textbf{Step 2}. Replacing $X$ by $\varphi X$ in \eqref{E-3.4} and using \eqref{E-nS-2.2}, we have
\begin{align}\label{E-3.13}
\nonumber
 & -g(R_{\,Q X,Y}Z, V) {+}\eta(X) g(R_{\,\xi,Y}Z, V) {+}g(R_{\,\varphi X, \varphi Y}Z, V)
 {+}g(R_{\,\varphi X, Y}\,\varphi Z, V) {+}g(R_{\,\varphi X, Y}Z, \varphi V) \\
\nonumber
 & = g(Y,V)\,g((h-\varphi)\varphi X,Z)-g(\varphi X,Y)\,g(Z,(h-\varphi)V) + g(Y,Z)\,g(\varphi X, (h-\varphi)V) \\
 &-(1/2)\,g(Z,V)\,g((h-\varphi)\varphi X,Y) +(1/2)\,g(\varphi X,Z)\,g(Y, (h-\varphi)V) .
\end{align}
Exchanging $X$ and $Y$ in \eqref{E-3.13}, we find
\begin{align}\label{E-3.14}
\nonumber
 &  g(R_{\,X, Q Y}Z, V) {+}\eta(Y) g(R_{\,\xi,X}Z, V) {-} g(R_{\,\varphi X, \varphi Y}Z, V)
 {+} g(R_{\,\varphi Y, X}\,\varphi Z, V) {+} g(R_{\,\varphi Y, X}Z, \varphi V) \\
\nonumber
 & = g(X,V)\,g((h-\varphi)\varphi Y,Z)-g(\varphi X,Y)\,g(Z, (h-\varphi)V) - g(X,Z)\,g(\varphi(Y, h-\varphi)V) \\
 &-(1/2)\,g(Z,V)\,g(X, (h-\varphi)\varphi Y) +(1/2)\,g(\varphi Y,Z)\,g(X, (h-\varphi)V) .
\end{align}
Subtracting \eqref{E-3.14} from \eqref{E-3.13}, we obtain
\begin{align}\label{E-3.15}
\nonumber
 & 2\,g(R_{\,\varphi X,\varphi Y}Z, V) -2\,g(R_{\,X,Y}Z, V) +\eta(X)\,g(R_{\,\xi,Y}Z, V) -\eta(Y)\,g(R_{\,\xi,X}Z, V) \\
\nonumber
 &\ + g(R_{\,\varphi X, Y}\,\varphi Z, V) - g(R_{\,\varphi Y, X}\,\varphi Z, V) + g(R_{\,\varphi X, Y}Z, \varphi V) - g(R_{\,\varphi Y, X}Z, \varphi V) \\
\nonumber
 &\ - g(R_{\widetilde Q X, Y}Z, V) - g(R_{\,X, \widetilde Q Y}Z, V) \\
 \nonumber
 & = g(Y,V)\,g((h-\varphi)\varphi X,Z) + g(Y,Z)\,g((h-\varphi)V,\varphi X) + g(Z,V)\,g(\widetilde Q X,Y) \\
 \nonumber
 &\ +(1/2)\,g(\varphi X,Z)\,g(Y, (h-\varphi)V) -(1/2)\,g(\varphi Y,Z)\,g(X, (h-\varphi)V) \\
 &\ -g(X,V)\,g((h-\varphi)\varphi Y,Z) + g(X,Z)\,g(Y, \varphi(h-\varphi)V) .
\end{align}
Then, replacing $Z$ by $\varphi Z$ and also $V$ by $\varphi V$ in \eqref{E-3.4} and using \eqref{E-nS-2.2}, we get two equations
\begin{align}\label{E-3.16}
\nonumber
 & -g(R_{\,X, Y}Q Z, V) = {-}\eta(Z)\,g(R_{\,\xi, V}\,X,Y) {-} g(R_{\,X,Y}\,\varphi Z, \varphi V)
 {-} g(R_{\,X,\varphi Y}\,\varphi Z, V) {-} g(R_{\,\varphi X,Y}\,\varphi Z, V)\\
\nonumber
 &\ - g(Y,V)\,g(\varphi(h-\varphi)X, Z)-g(X,Y)\,g(\varphi Z, (h-\varphi)V) + g(Y,\varphi Z)\,g(X, (h-\varphi)V) \\
 &\ -(1/2)\,g(\varphi Z,V)\,g((h-\varphi)X,Y) +(1/2)\,g(X,\varphi Z)\,g(Y, (h-\varphi)V),\\
\label{E-3.17}
\nonumber
 & -g(R_{\,X, Y}Z, Q V) = \eta(V)\,g(R_{\,\xi,Z}\,X,Y) {-} g(R_{\,X,Y}\,\varphi Z, \varphi V)
 {-} g(R_{\,\varphi X, Y}Z, \varphi V) {-} g(R_{\,X,\varphi Y}Z, \varphi V) \\
 \nonumber
 &\ + g(Y,\varphi V)\,g((h-\varphi)X,Z)-g(X,Y)\,g(Z, (h-\varphi)\varphi V) + g(Y,Z)\,g(X, (h-\varphi)\varphi V) \\
 &\ -(1/2)\,g(Z,\varphi V)\,g((h-\varphi)X,Y) +(1/2)\,g(X,Z)\,g(Y, (h-\varphi)\varphi V) .
\end{align}
Adding \eqref{E-3.16} to \eqref{E-3.17}, we get
\begin{align*}
 &-2\,g(R_{\,X, Y}Z, V) = -2\,g(R_{\,X,Y}\,\varphi Z, \varphi V) + g(R_{\,X, Y}\widetilde Q Z, V) + g(R_{\,X, Y}Z, \widetilde Q V) \\
 &\ -\,\eta(Z)\,g(R_{\,\xi, V}\,X,Y)
   - g(R_{\,X,\varphi Y}\,\varphi Z, V) - g(R_{\,\varphi X,Y}\,\varphi Z, V)\\
 &\ +\,\eta(V)\,g(R_{\,\xi, Z}\,X,Y) - g(R_{\,\varphi X, Y}Z, \varphi V) - g(R_{\,X,\varphi Y}Z, \varphi V) \\
\nonumber
 &\ - g(Y,V)\,g(\varphi(h-\varphi)X, Z) + g(Y,\varphi Z)\,g(X, (h-\varphi)V) \\
\nonumber
 &\ + g(Y,\varphi V)\,g((h-\varphi)X,Z) +2\,g(X,Y)\,g(Z, \widetilde Q V) + g(Y,Z)\,g(X, (h-\varphi)\varphi V) \\
 &\ +(1/2)\,g(X,\varphi Z)\,g(Y, (h-\varphi)V) +(1/2)\,g(X,Z)\,g(Y, (h-\varphi)\varphi V) .
\end{align*}
Substituting the above equation into \eqref{E-3.15}, we have
\begin{align}\label{E-3.18}
\nonumber
 & 2\,g(R_{\,\varphi X,\varphi Y}Z, V) -2\,g(R_{\,X,Y}\,\varphi Z, \varphi V) - \eta(Z)\,g(R_{\,\xi, V}X,Y) + \eta(V)\,g(R_{\,\xi,Z}\,X,Y) \\
\nonumber
 & +\,\eta(X)\,g(R_{\,\xi,Y}Z, V) -\eta(Y)\,g(R_{\,\xi,X}Z, V) +\delta(X,Y,Z,V) \\
\nonumber
 & - g(Y,V)\,g(\varphi(h-\varphi)X, Z) + g(Y,\varphi Z)\,g(X, (h-\varphi)V) \\
\nonumber
 & + g(Y,\varphi V)\,g((h-\varphi)X,Z) + g(Y,Z)\,g(X, (h-\varphi)\varphi V) + 2\,g(X,Y)\,g(Z, \widetilde Q V) \\
 &  +(1/2)\,g(X,\varphi Z)\,g(Y, (h-\varphi)V) +(1/2)\,g(X,Z)\,g(Y, (h-\varphi)\varphi V) = 0.
\end{align}
Using \eqref{E-3.23b} that is supporting by \eqref{E-nS-04c}, \eqref{E-3.18} and symmetry of $h^2$ and $Q$, we obtain \eqref{E-3.6}:
\begin{align*}
\nonumber
& 2\,g(R_{\,\varphi X,\varphi Y}Z, V) -2\,g(R_{\,X,Y}\,\varphi Z, \varphi V) + \delta(X,Y,Z,V)  + 2\,g(X,Y)\,g(Z, \widetilde Q V) \\
\nonumber
 & - g(Y,V)\,g(\varphi(h-\varphi)X, Z) + g(Y,\varphi Z)\,g(X, (h-\varphi)V) \\
\nonumber
 & + g(Y,\varphi V)\,g((h-\varphi)X,Z) + g(Y,Z)\,g(X, (h-\varphi)\varphi V) \\
 &  +(1/2)\,g(X,\varphi Z)\,g(Y, (h-\varphi)V) +(1/2)\,g(X,Z)\,g(Y, (h-\varphi)\varphi V) = 0.
\end{align*}
Replacing $X$ by $\varphi X$ and $Y$ by $\varphi Y$ in \eqref{E-3.6} and using \eqref{E-nS-2.2}, we get \eqref{E-3.5}.
\hfill$\square$

\smallskip\textbf{Proof of Lemma~\ref{L-R03}}.
In Step 1, our formulas depend on four vectors $X,Y,Z,V\in\mathfrak{X}_M$ and contain many terms.
In~Step 2, the~formulas depend only on three vectors from $TM$ and contain a relatively small number of terms.

\textbf{Step 1}. Differentiating \eqref{E-3.29} and using
$g((\nabla_X\varphi)(\nabla_V\varphi)Y, Z)=-g((\nabla_V\varphi)Y, (\nabla_X\varphi)Z)$ gives
\begin{align}\label{E-3.32}
\nonumber
 & g((\nabla_V\,\varphi)Y, (\nabla_X\,\varphi)Z) + g((\nabla_X\,\varphi)Y, (\nabla_V\,\varphi)Z) \\
\nonumber
 & = g((\nabla^2_{V,X}\,\varphi)\,\varphi Y, Z) + g((\nabla^2_{V,X}\,\varphi)\,\varphi Z, Y) - g(Y, (h-\varphi)V)\,g((h-\varphi)X, Z) \\
 &\ -\,\eta(Y)\,g((\nabla_V (h-\varphi))X,Z) -\nabla_V\big(\eta(Z)\,g((h-\varphi)X,\, Q Y) \big).
\end{align}
On the other hand, using \eqref{EF-nS-01}, \eqref{E-3.11} and \eqref{E-nS-2.2}, we find $\nabla^2$-terms in \eqref{E-3.32}:
\begin{align}
 \label{ER-nS-04a}
\nonumber
 & g((\nabla^2_{V,X}\,\varphi)\varphi Y, Z) = g((\nabla^2_{V, \varphi Y}\,\varphi)Z, X) \\
\nonumber
 &\ +2\,g(X,\varphi Y)\,g(Z, (h-\varphi)V) +g(X,Z)\,g(Y, \varphi(h-\varphi)V) -g(\varphi Y,Z)\,g(X, (h-\varphi)V) \\
\nonumber
 & = -g(R_{\,X,Z}V, Q Y) -\eta(Y)\,g(R_{\,\xi, V}\,X,Z) +(3/2)\,g(X,\varphi Y)\,g(Z, (h-\varphi)V) \\
\nonumber
 &\ + g(Z,V)\,g(\varphi(h-\varphi)X, Y) + (1/2)\,g(\varphi Y, V)\,g((h-\varphi)X,Z) \\
 &\ -(1/2)\,g(R_{\,X,V}\,\varphi Y, \varphi Z) - (1/2)\,g(R_{\,V,Z}\,\varphi Y, \varphi X) - (1/2)\,g(R_{\,X,Z}\,\varphi Y, \varphi V) ,\\
\label{ER-nS-04aa}
\nonumber
 & g((\nabla^2_{V,X}\,\varphi)\varphi Z, Y) = g((\nabla^2_{V, \varphi Z}\,\varphi)Y, X) \\
\nonumber
 &\ +2\,g(X,\varphi Z)\,g(Y, (h-\varphi)V) +g(X,Y)\,g(Z, \varphi(h-\varphi)V) -g(Y, \varphi Z)\,g(X, (h-\varphi)V) \\
\nonumber
 & = -g(R_{\,X,Y}V, Q Z) -\eta(Z)\,g(R_{\,\xi,V}\,X,Y) +(3/2)\,g(X,\varphi Z)\,g(Y,(h-\varphi)V) \\
\nonumber
 &\ + g(Y,V)\,g(\varphi(h-\varphi)X, Z) +(1/2)\,g(\varphi Z,V)\,g((h-\varphi)X,Y) \\
 &\ -(1/2)\,g(R_{\,X,V}\,\varphi Z, \varphi Y) - (1/2)\,g(R_{\,V,Y}\,\varphi Z, \varphi X) - (1/2)\,g(R_{\,X,Y}\,\varphi Z, \varphi V) .
\end{align}
Using \eqref{E-3.23b} and \eqref{ER-nS-04a}--\eqref{ER-nS-04aa}, we get from \eqref{E-3.32} the equality
\begin{align}\label{ER-nS-03b}
\nonumber
 & g((\nabla_V\,\varphi)Y, (\nabla_X\,\varphi)Z) + g((\nabla_X\,\varphi)Y, (\nabla_V\,\varphi)Z) = -g(R_{\,X,Z}V, Q Y) -\eta(Y)\,g(R_{\,\xi,V}\,X,Z) \\
\nonumber
 &\ + (1/2)g(R_{\,Y,V}\,\varphi Z, \varphi X)  {+} (1/2)g(R_{\,Y,X}\,\varphi Z, \varphi V)
 {-} (1/2)g(R_{\,V,Z}\,\varphi Y, \varphi X) {-} (1/2)g(R_{\,X,Z}\,\varphi Y, \varphi V)  \\
\nonumber
 & -g(R_{\,X,Y}V, Q Z) -\eta(Z)\,g(R_{\,\xi,V}\,X,Y) - g((h-\varphi)V,Y)\,g((h-\varphi)X, Z) \\
\nonumber
 &\ +(3/2)\,g(X,\varphi Y)\,g(Z,(h-\varphi)V) {+} g(Z,V)\,g(\varphi(h-\varphi)X, Y) {+}(1/2)\,g(\varphi Y,V)\,g((h-\varphi)X,Z) \\
\nonumber
 &\ +(3/2)\,g(X,\varphi Z)\,g(Y,(h-\varphi)V) {+} g(Y,V)\,g(\varphi(h-\varphi)X, Z) {+}(1/2)\,g(\varphi Z,V)\,g((h-\varphi)X,Y) \\
 & -\,\eta(Y)\,g((\nabla_V (h-\varphi))X,Z) -\nabla_V\big(\eta(Z)\,g((h-\varphi)X,\, Q Y) \big).
\end{align}
Using \eqref{E-3.6} for $g(R_{\,Y,V}\,\varphi Z, \varphi X)= g(R_{\,\varphi X, \varphi Z}V, Y)$
and $g(R_{\,Y,X}\,\varphi Z, \varphi V)= g(R_{\,\varphi V, \varphi Z}X, Y)$ in \eqref{ER-nS-03b},
and replacing $(\nabla_V (h-\varphi))X$ by \eqref{E-3.24} and using $g(R_{\,\varphi X,\,\varphi Y}\,Z,\,\xi) = 0$, see \eqref{E-nS-04cc}, we get
\begin{align}\label{E-3.34}
\nonumber
 & g((\nabla_V\,\varphi)Y, (\nabla_X\,\varphi)Z) + g((\nabla_X\,\varphi)Y, (\nabla_V\,\varphi)Z) = -g(R_{\,X,Z}V, Q Y)  -g(R_{\,X,Y}V, Q Z) \\
\nonumber
 &\ + g(R_{\,V,Z}\,\varphi X, \varphi Y) + g(R_{\,X,Z}\,\varphi V, \varphi Y) - g(Y, (h-\varphi)V)\,g((h-\varphi)X, Z) \\
\nonumber
 &\ -\eta(X)\eta(Z)\,\,g(Y, (h-\varphi)^2 V) + \eta(Y)\,\eta(Z)\,g(X, (h-\varphi)^2 V) \\
\nonumber
 &\ -(1/2)\,g(Y,Z)\,g( \widetilde Q X + \varphi^2 X, V) + g(Y,V)\,g(\varphi(h-\varphi)X, Z) \\
\nonumber
 &\ + (3/4)\,g(Z,V)\,g(\varphi(h-\varphi)X,Y) - (1/4)\,g(X,Z)\,g((h-\varphi)\varphi Y,V) \\
\nonumber
 &\ -(1/4)\,g(X,V)\,g((h-\varphi)\varphi Y,Z) - g(Z,V)\,g(\widetilde Q Y, X) - (1/2)g(X,Z)\,g(\widetilde Q Y, V) \\
\nonumber
 &\ -(5/4)\,g(X,\varphi Z)\,g((h-\varphi)Y,V) -(1/4)\,g(Z,\varphi V)\,g((h-\varphi)X,Y) \\
\nonumber
 &\ +(3/2)\,g(X,\varphi Y)\,g((h-\varphi)V, Z) {-}(1/2)\,g(Y,\varphi V)\,g((h-\varphi)X,Z) \\
 &\ -(1/4)\,\delta(V,Z,X,Y) -(1/4)\,\delta(X,Z,V,Y) -\nabla_V\big(\eta(Z)\,g((h-\varphi)X,\, Q Y) \big).
\end{align}
Replacing $Z$ and $V$ by $\varphi Z$ and $\varphi V$ in \eqref{E-3.34}, we find
\begin{align}\label{E-3.35}
\nonumber
 & g((\nabla_{\varphi V}\,\varphi)Y, (\nabla_X\,\varphi)\varphi Z) {+} g((\nabla_X\,\varphi)Y, (\nabla_{\varphi V}\,\varphi)\varphi Z) =
 g(R_{\,X,\varphi Z}\,\varphi^2 V, \varphi Y) {-} g(R_{\,X,\varphi Z}\,\varphi V, Q Y) \\
\nonumber
 &\ - g(R_{\,X,Y}\,Q \varphi Z, \varphi V) + g(R_{\,\varphi X, \varphi Y}\,\varphi Z,\varphi V) - g(Y, (h-\varphi)\varphi V)\,g((h-\varphi)X, \varphi Z) \\
\nonumber
 &\ -(1/2)\,g(Y,\varphi Z)\,g(\widetilde Q X + \varphi^2 X, \varphi V) + g(Y,\varphi V)\,g(\varphi(h-\varphi)X, \varphi Z) \\
\nonumber
 &\ - (3/4)\,g(\varphi^2 Z, V)\,g(\varphi(h-\varphi)X, Y) - (1/4)\,g(X,\varphi Z)\,g((h-\varphi)\varphi Y,\varphi V) \\
\nonumber
 &\ -(1/4)\,g(X,\varphi V)\,g((h{-}\varphi)\varphi Y,\varphi Z) {+} g(\varphi^2 Z, V)\,g(X, \widetilde Q Y)
 {-} (1/2)g(X,\varphi Z)\,g(\widetilde Q Y, \varphi V) \\
\nonumber
 &\ -(5/4)\,g(X,\varphi^2 Z)\,g((h-\varphi)Y, \varphi V) +(1/4)\,g(\varphi Z, Q V)\,g((h-\varphi)X,Y) \\
\nonumber
 &\ +(3/2)\,g(X,\varphi Y)\,g((h-\varphi)\varphi V, \varphi Z) {-}(1/2)\,g(Y,\varphi^2 V)\,g((h-\varphi)X,\varphi Z) \\
 &\ -(1/4)\,\delta(\varphi V,\varphi Z,X,Y) -(1/4)\,\delta(X,\varphi Z,\varphi V,Y) .
\end{align}
Using \eqref{E-3.29}, \eqref{E-3.30}, \eqref{E-3.31}, \eqref{E-nS-2.2} and Lemma~\ref{L-nS-02}, we rewrite terms in the lhs of \eqref{E-3.35}:
\begin{align}\label{E-3.36}
\nonumber
& g((\nabla_{\varphi V}\,\varphi)Y, (\nabla_X\,\varphi)\varphi Z) =
 g(Q(\nabla_X\,\varphi)Z, (\nabla_{V}\,\varphi)Y) - \eta(V) g((\nabla_X\,\varphi)Z, \varphi hY) \\
\nonumber
&\ - 2\,\eta(Y) g((\nabla_X\,\varphi)Z, \varphi^2 V) + \eta(Z) g((\nabla_V\,\varphi)Y, \varphi(h-\varphi)X)  \\
\nonumber
&\ -\eta(Z)\eta(V)\,g((h-\varphi)X, (h-\varphi)Y) -2\,\eta(Z)\eta(Y)\,g((h-\varphi)X, \varphi V)\\
&\ -\,g(\varphi(h-\varphi)X, Z) g(Y, \varphi(h-\varphi)V) + g(Y, (h-\varphi)\varphi V)\,g((h-\varphi)X, Q Z),\\
\label{E-3.37}
\nonumber
& g((\nabla_X\,\varphi)Y, (\nabla_{\varphi V}\,\varphi)\varphi Z) =
 \eta(Z)\,g((\nabla_X\,\varphi)Y, (h{-}\varphi)\varphi V) {-} g((\nabla_{V}\,\varphi)Z, Q(\nabla_X\,\varphi)Y) \\
\nonumber
&\ +\,\eta(V)\,g((\nabla_X\,\varphi)Y, \varphi h Z) - g(\varphi(h-\varphi)X, Y)\,g(\varphi(h-\varphi)V, Z) \\
&\ +\,g(\varphi(h-\varphi)X, Y)\,g((h-\varphi)\varphi V, Q Z).
\end{align}
From \eqref{E-3.23b} and Lemma~\ref{L-nS-02}, we have
\begin{align}\label{E-3.38}
\nonumber
 & g(R_{\,X, \varphi Z}\,\varphi V, Y) - g(R_{\,X, \varphi Z}\,\varphi^2 V, \varphi Y) = g(R_{\,X, \varphi Z}\,\varphi V, Y) \\
 & +g(R_{\,X, \varphi Z} Q V, \varphi Y) -\eta(X)\,\eta(V)\,g((h-\varphi)^2\varphi Y, \varphi Z).
\end{align}
On the other hand, from \eqref{E-3.4} and \eqref{E-3.6} it follows that
\begin{align}\label{E-3.39}
\nonumber
 & g(R_{\,X, Z}\,\varphi V, \varphi Y) + g(R_{\,X, Z}\,\varphi^2 V, Y) + g(R_{\,X, \varphi Z}\,\varphi V, Y) + g(R_{\,\varphi X, Z}\,\varphi V, Y) \\
\nonumber
 &\ = -g(Y,Z)\,g(\varphi(h-\varphi)X, V) + g(X,Z)\,g(\varphi(h-\varphi)Y, V) + g(Z,\varphi V)\,g(X, hY{-}\varphi Y) \\
 &\ -(1/2)\,g(Y,\varphi V)\,g((h-\varphi)X,Z) +(1/2)\,g(X,\varphi V)\,g((h-\varphi)Y,Z) ,\\
\label{E-3.40}
\nonumber
 & g(R_{\,\varphi X, Z}\,\varphi V, \varphi^2 Y)  = g(R_{\,\varphi^2 X, \varphi Z}V, \varphi Y) {+} (1/2)\delta(\varphi X,Z,V,\varphi Y)
 {+} g(\varphi X,Z)\,g(\widetilde Q \varphi Y, V)\\
\nonumber
 &\ +(1/2)\,g(\varphi Y,Z)\,g((h-\varphi)\varphi X, \varphi V) + (1/2)\,g(Z,\varphi V)\,g(\varphi X, (h-\varphi)\varphi Y) \\
\nonumber
 &\ + (1/2)\,g(\varphi^2 Y, Z)\,g((h-\varphi)\varphi X, V) + (1/2)\,g(Z,V)\,g((h-\varphi)\varphi X, Q Y) \\
 &\ -(1/4)\,g(\varphi^2 X, V)\,g((h-\varphi)\varphi Y, Z) + (1/4)\,g(\varphi X,V)\,g(Q Y, (h-\varphi)Z) .
\end{align}
From \eqref{E-3.40}, using \eqref{E-nS-2.2},  we find
\begin{align}\label{ER-nS-08bb}
\nonumber
 & -g(R_{\,\varphi X,Z}\,\varphi V, Q Y) +\eta(Y)\,g(R_{\,\varphi X, Z}\,\varphi V,\xi) =-g(R_{\,Q X,\varphi Z}V,\varphi Y) \\
\nonumber
 &\ +\eta(X)\,g(R_{\,\xi,\varphi Z}V,\varphi Y) +(1/2)\,\delta(\varphi X,Z,V,\varphi Y) {+} g(\varphi X,Z)\,g(\widetilde Q \varphi Y, V)\\
\nonumber
 &\ +(1/2)\,g(\varphi Y,Z)\,g((h-\varphi)\varphi X, \varphi V) + (1/2)\,g(Z,\varphi V)\,g(\varphi X, (h-\varphi)\varphi Y) \\
\nonumber
 &\ + (1/2)\,g(\varphi^2 Y, Z)\,g((h-\varphi)\varphi X, V) + (1/2)\,g(Z,V)\,g((h-\varphi)\varphi X, Q Y) \\
 &\ -(1/4)\,g(\varphi^2 X, V)\,g((h-\varphi)\varphi Y, Z) + (1/4)\,g(\varphi X,V)\,g(Q Y, (h-\varphi)Z) .
\end{align}
Summing up the formulas \eqref{E-3.39} and \eqref{ER-nS-08bb}
(and using \eqref{E-3.39}, \eqref{E-3.40}, \eqref{E-3.23b} and Lemma~\ref{L-nS-02}), we get
\begin{align}\label{E-3.41}
\nonumber
 & g(R_{\,\varphi X, Z}\,\varphi V, Y) +g(R_{\,Q X,\varphi Z}V,\varphi Y) = - g(R_{\,X, Z}\,\varphi V, \varphi Y) - g(R_{\,X, Z}\,\varphi^2 V, Y) \\
\nonumber
 &  - g(R_{\,X, \varphi Z}\,\varphi V, Y) + g(R_{\,\varphi X,Z}\,\varphi V, Q Y) -\eta(Y)\,\eta(Z)\,g(\varphi X, (h-\varphi)^2\varphi V) \\
\nonumber
 &\ - g(Y,Z)\,g(\varphi(h-\varphi)X, V) +g(X,Z)\,g(\varphi(h-\varphi)Y, V) + g(Z,\varphi V)\,g(X, (h-\varphi)Y) \\
\nonumber
 &\ -(1/2)\,g(Y,\varphi V)\,g((h-\varphi)X,Z) +(1/2)\,g(X,\varphi V)\,g((h-\varphi)Y,Z) \\
\nonumber
&\ +\eta(X)\,\eta(V)\,g(\varphi Y, (h-\varphi)^2\varphi Z)
{+}(1/2)\,\delta(\varphi X,Z,V,\varphi Y) {+} g(\varphi X,Z)\,g(\widetilde Q \varphi Y, V)\\
\nonumber
 &\ +(1/2)\,g(\varphi Y,Z)\,g((h-\varphi)\varphi X, \varphi V) + (1/2)\,g(Z,\varphi V)\,g(\varphi X, (h-\varphi)\varphi Y) \\
\nonumber
 &\ + (1/2)\,g(\varphi^2 Y, Z)\,g((h-\varphi)\varphi X, V) + (1/2)\,g(Z,V)\,g((h-\varphi)\varphi X, Q Y) \\
 &\ -(1/4)\,g(\varphi^2 X, V)\,g((h-\varphi)\varphi Y, Z) + (1/4)\,g(\varphi X,V)\,g(Q Y, (h-\varphi)Z) .
\end{align}
Substituting \eqref{E-3.41} into \eqref{E-3.38}, we get
\begin{align}\label{E-3.42}
\nonumber
 & g(R_{\,X, \varphi Z}\,\varphi V, Y) - g(R_{\,X, \varphi Z}\,\varphi^2 V, \varphi Y)
  = - g(R_{\,Q X,\varphi Z}V,\varphi Y) - g(R_{\,X, Z}\,\varphi V, \varphi Y)  \\
\nonumber
 &\ - g(R_{\,X, Z}\,\varphi^2 V, Y) - g(R_{\,X, \varphi Z}\,\varphi V, Y) + g(R_{\,\varphi X,Z}\,\varphi V, Q Y) +g(R_{\,X, \varphi Z} Q V, \varphi Y) \\
\nonumber
 &\ + g(Y,Z)\,g((h-\varphi)X,\varphi V) + g(X,Z)\,g(\varphi(h-\varphi)Y, V) + g(Z,\varphi V)\,g(X, (h-\varphi)Y) \\
\nonumber
 &\ -(1/2)\,g(Y,\varphi V)\,g((h-\varphi)X,Z) +(1/2)\,g(X,\varphi V)\,g((h-\varphi)Y,Z) \\
\nonumber
&\ +(1/2)\,\delta(\varphi X,Z,V,\varphi Y) {+} g(\varphi X,Z)\,g(\widetilde Q \varphi Y, V)
 -\eta(Y)\,\eta(Z)\,g((h-\varphi)^2\varphi V, \varphi X) \\
\nonumber
 &\ +(1/2)\,g(\varphi Y,Z)\,g((h-\varphi)\varphi X, \varphi V) + (1/2)\,g(Z,\varphi V)\,g(\varphi X, h\varphi Y {+} Q Y) \\
\nonumber
 &\ + (1/2)\,g(\varphi^2 Y, Z)\,g((h-\varphi)\varphi X, V) + (1/2)\,g(Z,V)\,g((h-\varphi)\varphi X, Q Y) \\
 &\ -(1/4)\,g(\varphi^2 X, V)\,g((h-\varphi)\varphi Y, Z) + (1/4)\,g(\varphi X,V)\,g(Q Y, (h-\varphi)Z) .
\end{align}
By means of \eqref{E-3.23b}, from \eqref{E-3.6} and \eqref{E-3.5}, we have
\begin{align}\label{E-3.43}
\nonumber
 & g(R_{\,\varphi X,\varphi Y}\,\varphi Z, \varphi V) - g(R_{\,X,Y}\,\varphi Z, \varphi V)
 = - g(R_{\,\varphi X,\varphi Y}Z, V) + g(R_{\,Q X, Q Y}Z, V) \\
\nonumber
 & {-}\eta(X)\,g(R_{\,\xi,Q Y}Z, V) + \eta(Y)\,g(R_{\,\xi, Q X}Z, V) \\
\nonumber
 &\ +(1/2)\,g(\varphi Y,V)\,g((h-\varphi)\varphi X,\varphi Z) + (1/2)\,g(\varphi^2 Y, Z)\,g(X, \varphi(h-\varphi)V) \\
\nonumber
 &\ - (1/2) g(\varphi^2 Y,V) g((h-\varphi)\varphi X,Z) + (1/2) g(\varphi Y,Z) g(\varphi X, (h-\varphi)\varphi V) \\
\nonumber
 &\ +(1/4)\,g(\varphi^2 X, Z)\,g(Y, \varphi(h-\varphi)V) +(1/4)\,g(\varphi X,Z)\,g(\varphi Y, (h-\varphi)\varphi V) \\
\nonumber
 &\ +(1/2)\,g(Y,V)\,g(\varphi(h-\varphi)X, Z) - (1/2)\,g(Y,\varphi Z)\,g(X, (h-\varphi)V) \\
\nonumber
 &\ - (1/2)\,g(Y,\varphi V)\,g((h-\varphi)X,Z) - (1/2)\,g(Y,Z)\,g(X, (h-\varphi)\varphi V) \\
\nonumber
 &\ -(1/4)\,g(X,\varphi Z)\,g(Y, (h-\varphi)V) -(1/4)\,g(X,Z)\,g(Y, (h-\varphi)\varphi V) \\
 &\ - g(\varphi^2 X + X, Y)\,g(\widetilde Q V, Z)
  +(1/2)\,\delta(\varphi X,\varphi Y,Z,V) -(1/2)\,\delta(X,Y,Z,V) .
\end{align}

\textbf{Step 2}. Substituting \eqref{E-3.36}, \eqref{E-3.37}, \eqref{E-3.42} and \eqref{E-3.43} into \eqref{E-3.35} and
adding the result to \eqref{E-3.34}, we obtain a multi-term equation (which we do not place here).
Then putting $\xi$ on $V$, we get the following equation with two $\nabla\varphi$-terms and a relatively small number of other terms:
\begin{align*}
\nonumber 
 & g((\nabla_X\,\varphi)Z, \widetilde Q Y + Q(\varphi h +\widetilde Q)Y)
 + g((\nabla_X\,\varphi)Y, (2\,\varphi h +\widetilde Q)Z - Q(\varphi h +\widetilde Q)Z) \\
\nonumber
&\ - \eta(Z)\,g((h-\varphi)X, \varphi(\varphi h +\widetilde Q)Y + (h-\varphi)Y) = - g(R_{\,X,Z}\,\xi, Q Y) \\
\nonumber
&\ - g(R_{\,X,Y}\,\xi, Q Z) + g(R_{\,\xi,Z}\,\varphi X, \varphi Y) - \eta(X)\,g(R_{\,\xi,Q Y}Z, \xi) + \eta(Y)\,g(R_{\,\xi, Q X}Z, \xi)\\
\nonumber 
&\  + g(R_{\,Q X,\varphi Z}\,\xi,\varphi Y) - g(R_{\,X, \varphi Z}\,\xi, \varphi Y) - g(R_{\,\varphi X,\varphi Y}Z, \xi) + g(R_{\,Q X, Q Y}Z, \xi) \\
\nonumber
&\ + \eta(Y)\,g(\varphi(h-\varphi)X, Z) + (3/4)\,\eta(Z)\,g(\varphi(h-\varphi)X,Y) -(1/4)\,\eta(X)\,g((h-\varphi)\varphi Y,Z) \\
\nonumber
&\ - \eta(Z)\,g(\widetilde Q Y, X) - (1/2)\,\eta(Z)\,g((h-\varphi)\varphi X, Q Y) +(1/2)\,\eta(Y)\,g(\varphi(h-\varphi)X, Z) .
\end{align*}
After applying \eqref{E-3.23b} the above equation is reduced to the following form:
\begin{align}\label{E-3.45-xi}
\nonumber
& g((\nabla_X\,\varphi)Z, Q\varphi(h{-}\varphi)Y - Y)
 + g((\nabla_Y\,\varphi)X , \widetilde Q\varphi(h-\varphi)Z -  \widetilde Q Z -\varphi h Z) \\
&\ +\eta(X)\,g(p_1(Z),Y) +\eta(Y)\,g(p_2(Z),X)+\eta(Z)\,g(p_3(X),Y) = 0 ,
\end{align}
where $p_1,p_2,p_3$ are linear operators (polynomials in $h,\varphi,\widetilde Q$) on $TM$ and $p_1,p_2$ vanish on $\xi$.

Swapping $X\leftrightarrow Z$ and $V\leftrightarrow Y$ in the multi-term equation (mentioned above),
then putting $\xi$ on $V$, we get another equation with two $\nabla\varphi$-terms and a relatively small number of other terms:
\begin{align*}
\nonumber 
 & -g((\nabla_Z\,\varphi)X, 2\,\varphi hY+\widetilde Q\varphi(h-\varphi)Y) + g((\nabla_{Y}\,\varphi)X,\, \widetilde Q\varphi(h-\varphi)Z) \\
\nonumber
&\ + \eta(X)\,g( 2\,\varphi^2 Y + 2\,Q Y , \varphi(h-\varphi)Z) - \eta(Y)\,g(\varphi(h-\varphi)Z, \varphi h X) = - g(R_{\,Z,X}Y, \xi)\\
\nonumber
&\ + g(R_{\,\xi, Z}Y, Q X) + g(R_{\,Z, X}\,\varphi^2 Y, \xi) + g(R_{\,Z, \varphi X}\,\varphi Y, \xi) - g(R_{\,\varphi Z,X}\,\varphi Y, \xi) \\
\nonumber
&\ - g(R_{\,\xi,Z}\,\varphi Y, \widetilde Q \varphi X) + \eta(X)\,g(Z, (h-\varphi)^2 Y) - (1/2)\,\eta(X)\,g( \widetilde Q Z + \varphi^2 Z, Y)  \\
\nonumber
&\ - \eta(X)\,g((h-\varphi)Z,\varphi Y) + \eta(X)\,g((h-\varphi)^2\varphi Y, \varphi Z) + \eta(Y)\,g(\varphi(h-\varphi)Z, X) \\
\nonumber
&\ +(1/2)\,\eta(Y)\,g(\varphi(h-\varphi)Z, X) - (1/2)\,\eta(X)\,g(Z, (h-\varphi)\varphi Y) - \eta(Z)\,g(\widetilde Q Y, X) .
\end{align*}
After applying \eqref{E-3.23b} and \eqref{E-nS-00b} the above equation is reduced to the following form:
\begin{align}\label{E-3.45b-xi}
\nonumber
&  g((\nabla_X\,\varphi)Z, 2\,\varphi hY+\widetilde Q\varphi(h-\varphi)Y) + g((\nabla_{Y}\,\varphi)X, \widetilde Q\varphi(h-\varphi)Z) \\
&\ +\eta(X)\,g(q_1(Z),Y) +\eta(Y)\,g(q_2(Z),X)+\eta(Z)\,g(q_3(X),Y) = 0 ,
\end{align}
where $q_1,q_2,q_3$ are linear operators (polynomials in $h,\varphi,\widetilde Q$) on $TM$ and $q_1,q_2$ vanish on $\xi$.

Deducing \eqref{E-3.45-xi} and \eqref{E-3.45b-xi}, we also used Lemma~\ref{L-nS-02}, \eqref{E-nS-04cc},
Remark~\ref{Rem-delta} (that all $\delta$-terms vanish) and $\nabla_\xi\big( \eta(Z)\,g(h X - \varphi X, \widetilde Q Y) \big)=0$.
Subtracting \eqref{E-3.45b-xi} from \eqref{E-3.45-xi}, we get
\begin{align}\label{E-3.46}
\nonumber
& - g((\nabla_Y\,\varphi)X,\ \varphi h Z + \widetilde Q Z) - g((\nabla_X\,\varphi)Z,\ \varphi h Y - \widetilde Q Y) \\
&\ +\eta(X)\,g(r_1(Z),Y) +\eta(Y)\,g(r_2(Z),X)+\eta(Z)\,g(r_3(X),Y) = 0 ,
\end{align}
where $r_1,r_2,r_3$ are linear operators on $TM$ and $r_1,r_2$ vanish on $\xi$.

Swapping $X$ and $Y$ in \eqref{E-3.46} we get
\begin{align}\label{E-3.46XY}
\nonumber
& - g((\nabla_X\,\varphi)Y,\ \varphi h Z + \widetilde Q Z) - g((\nabla_Y\,\varphi)Z,\ \varphi h X - \widetilde Q X) \\
&\ +\eta(Y)\,g(r_1(Z),X) +\eta(X)\,g(r_2(Z),Y)+\eta(Z)\,g(r_3(Y),X) = 0 .
\end{align}
Adding \eqref{E-3.46} to the resulting equation \eqref{E-3.46XY} with the use of \eqref{E-nS-00b} and \eqref{E-3.23b}, we get
\begin{align}\label{E-3.48}
\nonumber
& - g((\nabla_X\,\varphi)Z,\ \varphi h Y - \widetilde Q Y) - g((\nabla_Y\,\varphi)Z,\ \varphi h X - \widetilde Q X) \\
&\ +\eta(X)\,g(a_1(Z),Y) +\eta(Y)\,g(a_2(Z),X)+\eta(Z)\,g(a_3(X),Y) = 0 ,
\end{align}
where $a_1,a_2,a_3$ are linear operators on $TM$ and $a_1,a_2$ vanish on $\xi$.

Swapping $X$ and $Z$ in \eqref{E-3.48} we get
\begin{align}\label{E-3.49a}
\nonumber
& - g((\nabla_Z\,\varphi)X,\ \varphi h Y - \widetilde Q Y) - g((\nabla_Y\,\varphi)X,\ \varphi h Z - \widetilde Q Z) \\
&\ +\eta(Z)\,g(a_1(X),Y) +\eta(Y)\,g(a_2(X),Z)+\eta(X)\,g(a_3(Z),Y) = 0 .
\end{align}
Adding the resulting equation \eqref{E-3.49a} to \eqref{E-3.46}, then applying \eqref{E-3.24}, we get equation of the form
\begin{align}\label{E-3.50x}
& - 2\,g((\nabla_Y\,\varphi)X,\ \varphi h Z) +\eta(X)\,g(b_1(Z),Y) +\eta(Y)\,g(b_2(X),Z) +\eta(Z)\,g(b_3(X),Y) = 0,
\end{align}
where $b_1,b_2,b_3$ are linear operators (polynomials in $h,\varphi,\widetilde Q$) on $TM$ and $b_1,b_2$ vanish on $\xi$.
Taking $Z=\xi$ in \eqref{E-3.50x}, we get $g(b_3(X),Y)=0$ for all $X,Y$. Hence $b_3=0$ and \eqref{E-3.50x} is reduced~to
\begin{align}\label{E-3.50}
 2\,g((\nabla_{X}\,\varphi)Y, \varphi h Z) = \eta(X)\,g(b_1(Z),Y) +\eta(Y)\,g(b_2(Z),X).
\end{align}
Replacing $Y$ by $\varphi Y$ in \eqref{E-3.50} and using \eqref{E-3.29} and \eqref{E-nS-2.2}, we obtain
\begin{align}\label{E-3.50b}
  g((\nabla_{X}\,\varphi)Y, h Q Z) = \eta(X)\,g(c_1(Z), Y) +\eta(Y)\,g(c_2(Z), X) +\eta(Z)\,g(c_3(X),Y),
\end{align}
where $c_1,c_2,c_3$ are linear operators, and $c_1,c_2$ vanish on $\xi$.
Taking $Z=\xi$ in \eqref{E-3.50b}, we get $c_3=0$.

Taking $X=\xi$ in \eqref{E-3.50b}, and using $(\nabla_\xi\,\varphi)X = \varphi h X + \widetilde Q X$,
we find $c_1 = -(\varphi h^2 + \widetilde Q h) Q$.
Taking $Y=\xi$ in \eqref{E-3.50b}, and using \eqref{E-nS-01c}, we find $c_2 = (\varphi h^2 - h + \widetilde Q h) Q$.
Thus, \eqref{E-3.50b} reads~as
\begin{align}\label{E-3.50y}
 g((\nabla_{X}\,\varphi)Y, h Q Z) = -\eta(X)\,g((\varphi h^2 + \widetilde Q h)Q Z, Y) +\eta(Y)\,g((\varphi h^2 - h +\widetilde Q h) Q Z, X).
\end{align}
Finally, replacing $QZ$ by $Z$ in \eqref{E-3.50y}, since $Q$ is nonsingular, we get \eqref{E-3.50c}.
\hfill$\square$

\section*{Conclusions}

Some properties of the curvature tensor and the tensors $\varphi$ and $h$ of nearly Sasakian manifolds (see \cite{C-MD-2016,NDY-2018}) were extended
to weak nearly Sasakian manifolds with the conditions \eqref{E-nS-10} and \eqref{E-nS-04c}.
Our results answer the question in the introduction and show that the weak nearly Sasakian structure allows us to
expand the theory of contact metric manifolds.
Our Theorem~\ref{T-2.2} characte\-rizes Sasakian manifolds in the class of weak almost contact metric manifolds with the condition \eqref{E-nS-10}
using the property \eqref{E-nS-Sas}.
According to our Theorem~\ref{Th-4.5}, the properties \eqref{E-nS-10} and \eqref{E-nS-04c} are sufficient for a weak nearly Sasakian manifold of dimension greater than five to be Sasakian, this generalizes Theorem~3.3 from~\cite{NDY-2018}, see also \cite[Theorem~4.9]{C-MD-2019}.
After \cite{C-MD-2016,C-MD-2019,MN-2023} we are looking forward to studying five-dimensional weak nearly Sasakian manifolds.
A more difficult task is to extend a theory such as twistor spinors, see~\cite{K-2023}, by considering non-singular skew-symmetric tensors instead of linear complex structures.
Based on applications of the nearly Sasakian structure, we expect that certain weak structure will be useful for geometry and theoretical~physics.

\end{document}